\documentclass[11pt]{article}

\usepackage[a4paper,margin=29mm]{geometry}
\usepackage[T1]{fontenc}
\usepackage{lmodern}
\usepackage{microtype}
\usepackage{amsmath,amssymb,amsthm,mathtools,mathrsfs}
\usepackage{booktabs}
\usepackage{array}
\usepackage{enumitem}
\usepackage{float}
\usepackage{xcolor}
\usepackage{listings}
\usepackage[hidelinks]{hyperref}
\usepackage{aliascnt}
\usepackage[nameinlink,capitalize,noabbrev]{cleveref}

\definecolor{linkblue}{RGB}{33,77,126}
\hypersetup{
  colorlinks=true,
  linkcolor=linkblue,
  citecolor=linkblue,
  urlcolor=linkblue,
  pdftitle={Primitive Quadratic Polynomials in Additive Coset Families: Root Conics and Norm Descent},
  pdfauthor={Juncheng Zhou and Hongfeng Wu}
}

\lstdefinestyle{sagesource}{
  language=Python,
  basicstyle=\ttfamily\scriptsize,
  numbers=left,
  numberstyle=\tiny\color{black!55},
  numbersep=7pt,
  xleftmargin=2.2em,
  framexleftmargin=1.6em,
  columns=fullflexible,
  keepspaces=true,
  showstringspaces=false,
  breaklines=true,
  breakatwhitespace=false,
  tabsize=4
}

\newtheorem{theorem}{Theorem}[section]
\newaliascnt{proposition}{theorem}
\newtheorem{proposition}[proposition]{Proposition}
\aliascntresetthe{proposition}
\newaliascnt{lemma}{theorem}
\newtheorem{lemma}[lemma]{Lemma}
\aliascntresetthe{lemma}
\newaliascnt{corollary}{theorem}
\newtheorem{corollary}[corollary]{Corollary}
\aliascntresetthe{corollary}

\theoremstyle{definition}
\newaliascnt{definition}{theorem}

\aliascntresetthe{definition}
\theoremstyle{remark}
\newaliascnt{remark}{theorem}

\aliascntresetthe{remark}
\newtheorem*{remark*}{Remark}

\newcommand{\F}{\mathbb F}

\newcommand{\PP}{\mathbb P}
\newcommand{\rad}{\operatorname{rad}}
\newcommand{\ord}{\operatorname{ord}}
\newcommand{\Nm}{\operatorname{N}}

\newcommand{\one}{\mathbf 1}
\newcommand{\Acal}{\mathcal A}
\newcommand{\Bcal}{\mathcal B}
\newcommand{\Ccal}{\mathcal C}
\newcommand{\Pcal}{\mathcal P}
\newcommand{\Rcal}{\mathcal R}
\newcommand{\phicore}{\vartheta}

\DeclareMathOperator{\supp}{supp}

\title{\textbf{Primitive Quadratic Polynomials in Additive Coset
Families}}
\author{Juncheng Zhou and Hongfeng Wu
\footnote{Corresponding author.}
\setcounter{footnote}{-1}
\footnote{E-mail addresses:
jczhoumath@gmail.com (J. Zhou), whfmath@gmail.com (H. Wu)}
\\
{College of Science, North China University of Technology, Beijing, China}\\
}
\date{}

\begin{document}
\maketitle

\begin{abstract}
Let \(q\) be an odd prime power. For \(\mu\in\mathbb F_q^\times\) and an additive coset \(\bar\alpha\in\mathbb F_{q^2}/\mathbb F_q\), consider the family
\[
\{x^2+\mu x-\alpha:\alpha\in\bar\alpha\},
\]
where each polynomial is regarded over its coefficient field \(\mathbb F_q(\alpha)\). We prove that, for
\[
q\notin{7,11,13,19,29,31,41,43},
\]
every such family contains a primitive polynomial.

For the zero coset, the result follows from Cohen's prescribed-trace theorem. For a nonzero coset, after a natural normalization the roots are parameterized by two smooth projective conics arising from the two \(q^2\)-Frobenius eigenspaces in \(\mathbb F_{q^4}\). Their affine \(\mathbb F_q\)-point counts, \(q-1\) and \(q+1\), correspond respectively to the reducible and irreducible members of the family. On the irreducible root conic, \(q^2\)-Frobenius induces a fixed-point-free involution on rational points. Passing to the quotient conic allows the relative norm of the root function to descend to a rational function. Tensor induction then yields order-sensitive character-sum bounds with constants \(6,8\) on the root conic and the sharper constants \(2,4\) after norm descent.

Combining these estimates with a double-core prime sieve and an exact residue-cover verification completes the finite range. As consequences, Gow and McGuire's Conjecture~1 holds for every odd prime power \(q\ne13\), while their Conjectures~2 and~3 hold for every odd prime power \(q>43\); the threshold \(43\) is sharp.\\

\end{abstract}

\section{Introduction}

An irreducible polynomial over a finite field is called
\emph{primitive} if one (equivalently, every) of its roots generates
the multiplicative group of the corresponding extension field; see
\cite[Chapter~3]{LidlNiederreiter}.

Finding primitive elements under additive restrictions is a classical
problem in number theory.  The underlying difficulty is structural:
primitivity is multiplicative,
whereas natural conditions on a polynomial or its roots are often
additive.

This interaction has been studied in several directions.  Cohen proved
a prescribed-trace theorem and studied primitive elements on affine
lines \cite{CohenTrace,CohenLines}.  Vega gives an explicit description
of primitive quadratics over a fixed finite field, while Zhu and Wu
study their binomial orders and coefficient criteria
\cite{Vega,ZhuWu}.  Sharma considers
higher-degree translates \(g(x)+\lambda\) with \(\lambda\) primitive
\cite{Sharma}.  The conjectures of Gow and McGuire
\cite{GowMcGuire} concern a different kind of uniformity: for every
nonzero linear coefficient and every additive coset of constant terms,
the associated family of monic quadratics should contain a primitive
member.

We formulate the problem intrinsically.  For
\(\mu\in\F_q^\times\) and a coset
\(\bar\alpha\in\F_{q^2}/\F_q\), set
\begin{equation}\label{eq:coset-family}
  \mathcal F_{\mu,\bar\alpha}
  =\{x^2+\mu x-\alpha:\alpha\in\bar\alpha\}.
\end{equation}
Primitivity is understood over the coefficient field.  The zero coset
follows from Cohen's prescribed-trace theorem
\cite[Theorem~1]{CohenTrace}.  We seek a uniform existence theorem for
the remaining families.

Set
\begin{equation}\label{eq:exceptional-q}
  \mathcal E=\{7,11,13,19,29,31,41,43\}.
\end{equation}
Our main result is the following uniform existence theorem.

\begin{theorem}[Main Theorem]\label{thm:main}
Let \(q\) be an odd prime power.  If \(q\notin\mathcal E\), then
\(\mathcal F_{\mu,\bar\alpha}\) contains a primitive member for every
\(\mu\in\F_q^\times\) and every
\(\bar\alpha\in\F_{q^2}/\F_q\).
\end{theorem}

The new content of \Cref{thm:main} is the uniform result for nonzero
cosets and the conic geometry that drives it.

We now explain the relation with the conjectures of Gow and McGuire
\cite[Conjectures~1--3]{GowMcGuire}.  Given
\(\alpha\in\F_{q^2}\setminus\F_q\), they consider the quadratic
family
\begin{equation}\label{eq:family}
  f_\lambda(x)=x^2+\mu x+\lambda-\alpha
  \qquad(\lambda\in\F_q).
\end{equation}
Writing \(\bar\alpha=\alpha+\F_q\), we have
\[
  f_\lambda(x)=x^2+\mu x-(\alpha-\lambda),
\]
and \(\alpha-\lambda\) runs through \(\bar\alpha\).  Thus their
formulation gives the following consequence.

\begin{corollary}[Gow--McGuire conjectures]\label{cor:GM}
Conjecture~1 holds for every odd prime power \(q\neq13\).
Conjectures~2 and~3 hold for every odd prime power \(q>43\), and this
threshold is sharp.
\end{corollary}

For a nonzero coset, put
\[
  D=\alpha-\alpha^q,
  \qquad \Delta=\frac{D}{\mu^2},
  \qquad w=\frac{x}{\mu}+\frac12.
\]
The equations for all members then become
\[
  \Phi(w):=w^2-w^{2q}=\Delta.
\]
The two \(q^2\)-Frobenius eigenspaces in \(\F_{q^4}\) are
two-dimensional over \(\F_q\).  On them, \(\Phi\) restricts to a split
and an anisotropic nondegenerate form.  The corresponding projective
conics have \(q-1\) and \(q+1\) affine points, which are exactly the
roots of the reducible and irreducible members.

Let \(\mathscr C_\Delta\) be the irreducible root conic.  Its root
function \(b\) gives the roots as \(\mu b(P)\), and the involution
\(P\mapsto\iota P\) pairs the two conjugate roots.  The quotient is a
second conic on which their relative norm is a rational function.  The
same geometry therefore records factorization, conjugation, and norm
descent.

For a multiplicative character \(\chi\) of \(\F_{q^4}^\times\), we
estimate
\[
  S_\chi
  =\sum_{P\in\mathscr C_\Delta(\F_q)}
       \chi\bigl(\mu b(P)\bigr)
\]
by tensor induction \cite{FuWan}.  Four zero lines on the root conic,
together with its points at infinity, give the constants \(6\) and
\(8\).  After norm descent, two quotient lines give the sharper
constants \(2\) and \(4\).

These uniform bounds feed into a double-core prime sieve
\cite{BaggerPunch}.  Projective scaling and an exact residue-cover test
settle the remaining finite range and produce precisely \(\mathcal E\).
Sections~\ref{sec:reduction}--\ref{sec:characters} develop the conic and
character-sum arguments; Sections~\ref{sec:sieve}--\ref{sec:algorithm}
complete the sieve and finite verification.

\section{The normalized root conics}\label{sec:reduction}

Throughout this section, fix
\(\alpha\in\F_{q^2}\setminus\F_q\) and write
\(\bar\alpha=\alpha+\F_q\).  Thus \(\bar\alpha\) is a nonzero coset in
\(\F_{q^2}/\F_q\).  Let \(\sigma\) denote \(q\)-Frobenius on
\(\F_{q^2}\).  The \(\F_q\)-linear map \(1-\sigma\) has kernel
\(\F_q\) and image
\[
  L=\{d\in\F_{q^2}:d^q=-d\}.
\]
It therefore identifies \(\F_{q^2}/\F_q\) with the one-dimensional
\(\F_q\)-space \(L\), and the image of \(\bar\alpha\) is
\[
  D:=\alpha-\alpha^q\in L^\times,
\]
which depends only on \(\bar\alpha\).

\subsection{Normalization and the two root conics}

Fix \(\mu\in\F_q^\times\) and define
\[
  f_\lambda(x):=x^2+\mu x+\lambda-\alpha
  \qquad (\lambda\in\F_q).
\]
The identity
\[
  \prod_{\lambda\in\F_q}(x+\lambda)=x^q-x
\]
gives
\begin{equation}\label{eq:factor-D}
\begin{aligned}
 \prod_{\lambda\in\F_q}f_\lambda(x)
  &=f_0(x)^q-f_0(x)\\
  &=-\left[
       \left(x+\frac{\mu}{2}\right)^2
       -\left(x+\frac{\mu}{2}\right)^{2q}-D
     \right].
\end{aligned}
\end{equation}
Set
\[
  w=\frac{x}{\mu}+\frac12.
\]
Then the root equation for the family becomes
\begin{equation}\label{eq:normalized-root-equation}
  \Phi(w):=w^2-w^{2q}=\Delta,
\end{equation}
where \(\Delta:=D/\mu^2\in L^\times\).
The polynomial \(\Phi(w)-\Delta\) is separable of
degree \(2q\).  If \(w\) is one of its roots, then applying
\(q\)-Frobenius gives
\[
  w^{2q}-w^{2q^2}=\Delta^q=-\Delta=w^{2q}-w^2.
\]
Thus \(w^{2q^2}=w^2\), and hence
\[
  w^{q^2}=\pm w.
\]
Consequently all roots lie in \(\F_{q^4}\), and they split between the
two \(\F_q\)-planes
\[
  V_+=\{w\in\F_{q^4}:w^{q^2}=w\}=\F_{q^2},
  \qquad
  V_-=\{w\in\F_{q^4}:w^{q^2}=-w\}.
\]
Since \(q\) is odd, \(\F_{q^4}=V_+\oplus V_-\).

For \(w\in V_\pm\), we have \(w^2\in\F_{q^2}\), and therefore
\[
  \Phi(w)=w^2-(w^2)^q\in L.
\]
Hence the restrictions of \(\Phi\) to \(V_+\) and \(V_-\) are
\(L\)-valued quadratic forms over \(\F_q\).  To determine their types,
observe that, for any \(w\in\F_{q^4}\),
\[
  \Phi(w)=(w-w^q)(w+w^q).
\]
On \(V_+=\F_{q^2}\), the two factors vanish precisely on the distinct
\(\F_q\)-lines \(\F_q\) and \(L\), respectively.  On \(V_-\), however,
the vanishing of either factor implies \(w^{q^2}=w\); comparing this with
\(w^{q^2}=-w\) gives \(w=0\).
Thus both restrictions are nondegenerate: the one on \(V_+\) is split,
whereas the one on \(V_-\) is anisotropic.

Introduce a homogenizing coordinate \(u\).  For each sign, define the
projective conic in
\(\PP(\F_q\oplus V_\pm)\simeq\PP^2_{\F_q}\) by
\begin{equation}\label{eq:two-root-conics}
  \mathscr C_\Delta^\pm
  =
  \bigl\{[u:w]\in\PP(\F_q\oplus V_\pm):
             \Phi(w)=\Delta u^2\bigr\}.
\end{equation}
Since \(\Delta\neq0\) and both restrictions of \(\Phi\) are
nondegenerate, both conics are smooth.
Their points at infinity are precisely the isotropic directions of
\(\Phi|_{V_\pm}\).  Hence \(\mathscr C_\Delta^+\) has two
\(\F_q\)-rational points at infinity, whereas \(\mathscr C_\Delta^-\)
has none.  A smooth conic over \(\F_q\) with a rational point has
\(q+1\) rational points.  Therefore \(\mathscr C_\Delta^+(\F_q)\)
has \(q-1\) affine points.
Of the \(2q\) distinct roots of \(\Phi(w)-\Delta\), exactly \(q-1\)
lie in \(V_+\).  The remaining \(q+1\) lie in \(V_-\), and hence
\[
  |\mathscr C_\Delta^-(\F_q)|=q+1.
\]

The affine points of the two conics separate the reducible and
irreducible members.
For a root \(x=\mu(w-1/2)\), we have
\[
  x^{q^2}
  =\mu\left(w^{q^2}-\frac12\right)
  =
  \begin{cases}
    x,      & w\in V_+,\\
    -x-\mu, & w\in V_-.
  \end{cases}
\]
In the first case \(x\in\F_{q^2}\), so the corresponding member is
reducible.  In the second case, \(-x-\mu\) is the other root by Vieta's
formula.  Moreover, \(\Phi(w)=\Delta\neq0\) forces \(w\neq0\), so
\(x^{q^2}\neq x\).  Thus \(x\notin\F_{q^2}\), and the corresponding
member is irreducible.  This proves the following correspondence.

\begin{proposition}[The two root conics]\label{prop:root-conic}
Fix \(\mu\in\F_q^\times\).  The map
\[
  [1:w]\longmapsto \mu\left(w-\frac12\right)
\]
identifies the affine \(\F_q\)-points of \(\mathscr C_\Delta^+\) with
the roots of all reducible members of \eqref{eq:coset-family}, and all
\(\F_q\)-points of \(\mathscr C_\Delta^-\) with the roots of all
irreducible members.  In particular, the family has
\((q-1)/2\) reducible members and \((q+1)/2\) irreducible members.
\end{proposition}

\subsection{The irreducible root conic}

We henceforth write
\[
  \mathscr C_\Delta=\mathscr C_\Delta^-.
\]
For fixed \(\Delta\in L^\times\), the homogeneous equation
\[
  \Phi(w)=\Delta u^2
\]
defines the affine cone over
\(\mathscr C_\Delta\subset\PP^2_{\F_q}\).  Its slice \(u=\mu\) is the
centered root equation \(\Phi(w)=D\), where \(D=\Delta\mu^2\).  Thus
\(u\) is the multiplier coordinate on the cone, while \(w-u/2\) is the
root coordinate.  After base change to \(\F_{q^4}\), the
\(V_-\)-coordinate \(w\) is an \(\F_{q^4}\)-valued linear coordinate.
Define the normalized root function
\begin{equation}\label{eq:root-function}
  b\in\F_{q^4}(\mathscr C_\Delta)^\times,
  \qquad
  b([u:w])=\frac{w-u/2}{u}.
\end{equation}
On the slice \(u=\mu\), the corresponding root is
\[
  x=w-\frac{\mu}{2}=\mu b([\mu:w]).
\]
By \Cref{prop:root-conic}, the roots of the irreducible members for
\((\mu,D)\) are therefore precisely \(\mu b(P)\), with
\(P\in\mathscr C_\Delta(\F_q)\).  Every such point is affine, so \(b\)
is finite there; it is also nonzero, since
\(f_\lambda(0)=\lambda-\alpha\neq0\).

Thus scalar multiplication on the cone sends
\[
  (\mu,D,x)\longmapsto(c\mu,c^2D,cx)
  \qquad(c\in\F_q^\times).
\]
Let \(\mathcal R_{\mu,D}\) be the set of roots of the irreducible
members for \((\mu,D)\).  Then
\begin{equation}\label{eq:root-set-scaling}
  \mathcal R_{\mu,D}
  =\mu\,\mathcal R_{1,D/\mu^2}
  =\{\mu b(P):P\in\mathscr C_{D/\mu^2}(\F_q)\}.
\end{equation}
Allowing \(\mu\) to vary, for each \(\Delta\in L^\times\) the \(q-1\)
nonzero affine slices
\(u=\mu\) correspond exactly to the pairs
\((\mu,\Delta\mu^2)\), with \(\mu\in\F_q^\times\).  As \(\Delta\)
varies, these slices account for all \((q-1)^2\) pairs
\((\mu,D)\in\F_q^\times\times L^\times\).  The normalization \(\mu=1\)
selects the slice \(u=1\).

Returning to fixed \(\mu\) and \(\Delta\), let \(n=q^4-1\) and define
\begin{equation}\label{eq:Ndef}
  N(n)
   :=\#\{P\in\mathscr C_\Delta(\F_q):
             \mu b(P)\text{ is primitive in }\F_{q^4}^\times\}.
\end{equation}
Here and below, the dependence of \(N(n)\) on the fixed parameters
\(\mu\) and \(\Delta\) is suppressed.
By \Cref{prop:root-conic}, the elements counted by \(N(n)\)
are exactly the primitive roots of the irreducible members of the
family.  Consequently:

\begin{corollary}\label{cor:reduction}
The family \eqref{eq:coset-family} contains a primitive polynomial over
\(\F_{q^2}\) if and only if \(N(n)>0\).
\end{corollary}

\subsection{The quotient by the root-pair involution}
\label{subsec:root-pair-quotient}

Every point \([u:w]\in\mathscr C_\Delta(\F_q)\) has \(u\neq0\), and
\[
  b([u:w])^{q^2}=b([u:-w]).
\]
Thus \(\mu b([u:w])\) and \(\mu b([u:-w])\) are the two
Frobenius-conjugate roots of the same irreducible member, and the
root-pair involution is
\[
  \iota:\mathscr C_\Delta\longrightarrow\mathscr C_\Delta,
  \qquad [u:w]\longmapsto[u:-w].
\]

After base change to \(\overline{\F}_q\), the \(V_-\)-coordinate and
its \(q\)-Frobenius conjugate define linearly independent forms \(r\)
and \(s\), characterized on \(\F_q\)-points by
\[
  r([u:w])=w,\qquad s([u:w])=w^q.
\]
In these coordinates,
\[
  (\mathscr C_\Delta)_{\overline{\F}_q}:
  \quad r^2-s^2=\Delta u^2,
  \qquad
  \iota([u:r:s])=[u:-r:-s].
\]
Arithmetic Frobenius sends
\((u,r,s,\Delta)\) to \((u,s,-r,-\Delta)\).  Hence the invariant
quadrics
\[
  U=u^2,
  \qquad X=r^2+s^2,
  \qquad Y=\frac{2rs}{\Delta}
\]
are all defined over \(\F_q\).  Moreover,
\(\Delta^2\in\F_q^\times\), since \(\Delta^q=-\Delta\).  They satisfy
\[
  X^2-\Delta^2Y^2
  =(r^2+s^2)^2-4r^2s^2
  =(r^2-s^2)^2
  =\Delta^2U^2.
\]
Thus they define the \(\F_q\)-morphism
\begin{equation}\label{eq:root-quotient}
\begin{aligned}
  \pi_\Delta:\mathscr C_\Delta&\longrightarrow
  \overline{\mathscr C}_\Delta,\\
  [u:r:s]&\longmapsto
  \left[u^2:r^2+s^2:\frac{2rs}{\Delta}\right],
\end{aligned}
\end{equation}
where the quotient conic is
\begin{equation}\label{eq:quotient-conic}
  \overline{\mathscr C}_\Delta
  =\bigl\{[U:X:Y]\in\PP^2_{\F_q}:
          X^2-\Delta^2Y^2=\Delta^2U^2\bigr\}.
\end{equation}
The target coordinates determine the invariant quadrics on
\(\mathscr C_\Delta\):
\[
  u^2=U,
  \qquad r^2=\frac{X+\Delta U}{2},
  \qquad s^2=\frac{X-\Delta U}{2},
  \qquad rs=\frac{\Delta Y}{2}.
\]
Thus the generic fiber is
\(\{[u:r:s],[u:-r:-s]\}\), and \(\pi_\Delta\) realizes the quotient by
\(\langle\iota\rangle\).  It ramifies exactly at the two geometric
points at infinity \([0:1:\pm1]\), the fixed points of \(\iota\).

Although \(\pi_\Delta\) is the geometric quotient,
\(\overline{\mathscr C}_\Delta(\F_q)\) is not the orbit set.  The
involution acts freely on
\(\mathscr C_\Delta(\F_q)\), so
\(\mathscr C_\Delta(\F_q)/\langle\iota\rangle\) has \((q+1)/2\)
elements.  On the other hand, \(\pi_\Delta\) gives
\(\overline{\mathscr C}_\Delta\) an \(\F_q\)-point, and hence this smooth
conic has \(q+1\) rational points.

The quotient also records the product of the two roots in each pair.
If \(P=[u:r:s]\) is affine and
\(\pi_\Delta(P)=[U:X:Y]\), then
\begin{equation}\label{eq:norm-on-quotient}
\begin{aligned}
 \mu^2 b(P)b(\iota P)
 &=\mu^2\left(\frac14-\frac{r^2}{u^2}\right)\\
 &=\frac{\mu^2}{2}\,
   \frac{(\frac12-\Delta)U-X}{U}.
\end{aligned}
\end{equation}
For \(P\in\mathscr C_\Delta(\F_q)\), the left-hand side is
\(\Nm_{\F_{q^4}/\F_{q^2}}(\mu b(P))\).

\section{Character sums on conics}\label{sec:characters}

Throughout this section, fix \(\mu\in\F_q^\times\) and
\(\Delta\in L^\times\).  The bounds below will be uniform in both
parameters.

For a multiplicative character \(\chi\) of \(\F_{q^4}^\times\), extended
by \(\chi(0)=0\), set
\begin{equation}\label{eq:Schi}
  S_\chi
   =\sum_{P\in\mathscr C_\Delta(\F_q)}\chi\bigl(\mu b(P)\bigr).
\end{equation}
These are the character sums that enter the primitive-element count
in \eqref{eq:Ndef}.  The trivial character gives \(S_1=q+1\); both
estimates for nontrivial characters will follow from one projective
conic lemma.

\subsection{A divisor bound on a projective conic}

We first record the genus-zero consequence of the tensor-induction
formalism of Fu and Wan that will be used twice below.

Fix a prime \(\ell\neq\operatorname{char}\F_q\) and identify the
relevant complex roots of unity with roots of unity in
\(\overline{\mathbb Q}_\ell\).  We use the corresponding Kummer
sheaves and identify their trace values with the original character
values; see \cite[Sections~1--2]{FuWan}.

Let \(C\) be a smooth projective conic over \(\F_q\), let \(\chi\) be
a nontrivial multiplicative character of \(\F_{q^m}^\times\) of order
\(d\), and let \(f_0\in\F_{q^m}(C)^\times\).  Write
\[
  f_i=f_0^{[q^i]}\qquad(0\leq i<m),
\]
where the superscript means applying the \(q^i\)-power Frobenius to
the coefficients of \(f_0\) with respect to the \(\F_q\)-model \(C\),
and form
\begin{equation}\label{eq:tensor-product-function}
  G_f=f_0\prod_{i=1}^{m-1}f_i^{q^{m-i}}.
\end{equation}
For \(P\in C(\overline{\F}_q)\), put
\[
  e_P\equiv\ord_P(G_f)
  \equiv\ord_P(f_0)+\sum_{i=1}^{m-1}q^{m-i}\ord_P(f_i)
  \pmod d,
\]
and let
\[
  \Sigma_f=\{P:e_P\neq0\},\qquad R_f=|\Sigma_f|.
\]

\begin{theorem}[Projective-conic character-sum lemma]
\label{thm:conic-sum}
Assume that \(\Sigma_f\neq\varnothing\) and that \(f_0\) is finite and
nonzero on \(C(\F_q)\).  Then
\begin{equation}\label{eq:conic-sum-bound}
  \left|\sum_{P\in C(\F_q)}\chi(f_0(P))\right|
  \leq(R_f-2)\sqrt q.
\end{equation}
\end{theorem}

\begin{proof}
Remove from \(C_{\overline{\F}_q}\) the supports of
\(\operatorname{div}(f_i)\) for \(0\leq i<m\).  This open set is
Frobenius-stable and descends to an open subset \(V\subset C\).
Tensor induction of the Kummer sheaf \(\mathscr K_{\chi,f_0}\) gives a
rank-one sheaf \(\mathcal E\) on \(V\) whose trace at
\(P\in V(\F_q)\) is \(\chi(f_0(P))\)
\cite[Proposition~1.2]{FuWan}.  After base change to
\(\overline{\F}_q\), it is the Kummer sheaf attached to \(G_f\) in
\eqref{eq:tensor-product-function}
\cite[Proposition~2.1]{FuWan}.  Its tame inertia at \(P\) is therefore
nontrivial exactly when \(d\nmid\ord_P(G_f)\), so its ramification set
is \(\Sigma_f\), with Artin conductor one at each of these points.
At every point outside \(\Sigma_f\), the inertia character is trivial;
hence \(\mathcal E\) extends uniquely to
\(U=C\setminus\Sigma_f\).

Since \(\deg\operatorname{div}(G_f)=0\), the set \(\Sigma_f\) cannot
have one element.  Thus \(R_f\geq2\), and \(\mathcal E\) is
geometrically nonconstant.  Since \(U\) is affine,
\[
  H_c^0(U_{\overline{\F}_q},\mathcal E)
  =H_c^2(U_{\overline{\F}_q},\mathcal E)=0,
\]
and the Grothendieck--Ogg--Shafarevich formula \cite{MilneEtale} gives
\[
  \dim H_c^1(U_{\overline{\F}_q},\mathcal E)=R_f-2.
\]
The sheaf \(\mathcal E\) is pure of weight zero, so the trace formula
and Deligne's Riemann hypothesis \cite{DeligneWeilII} give
\eqref{eq:conic-sum-bound}.  At an \(\F_q\)-point,
\(f_i(P)=f_0(P)^{q^i}\); hence the hypothesis on \(f_0\) places every
point of \(C(\F_q)\) in \(U\), and the trace sum is complete.
\end{proof}

\subsection{Four lines on the root conic and the 6/8 estimate}

Let \(\chi\neq1\), and put \(d=\ord\chi\).  In the coordinates
\([u:r:s]\) of \Cref{subsec:root-pair-quotient}, define
\[
\begin{aligned}
  f_0&=\frac{r-u/2}{u},&
  f_1&=\frac{s-u/2}{u},&
  f_2&=\frac{-r-u/2}{u},&
  f_3&=\frac{-s-u/2}{u}.
\end{aligned}
\]
These satisfy \(f_i=f_0^{[q^i]}\).  On \(\F_q\)-points, \(b=f_0\), and
hence
\[
  S_\chi
  =\chi(\mu)\sum_{P\in\mathscr C_\Delta(\F_q)}\chi\bigl(f_0(P)\bigr).
\]
The prefactor has absolute value \(1\), so we apply
\Cref{thm:conic-sum} with
\[
  C=\mathscr C_\Delta,\qquad m=4
\]
and initial function \(f_0\).
Their zero divisors on
\[
  (\mathscr C_\Delta)_{\overline{\F}_q}:
  \quad r^2-s^2=\Delta u^2
  \quad\subset\PP^2_{\overline{\F}_q}
\]
are cut out by the four lines
\[
  L_0:r=\frac u2,\qquad
  L_1:s=\frac u2,\qquad
  L_2:r=-\frac u2,\qquad
  L_3:s=-\frac u2.
\]
On \(L_0\) and \(L_2\), the conic equation becomes
\(s^2=(1/4-\Delta)u^2\); on \(L_1\) and \(L_3\), it becomes
\(r^2=(1/4+\Delta)u^2\).  Since \(\Delta\notin\F_q\), both
\(1/4-\Delta\) and \(1/4+\Delta\) are nonzero.  Hence each line meets
the conic transversely in two affine geometric points.  These eight
points are pairwise distinct: opposite lines have no common point on
the conic, while a point on both an \(r\)-line and an \(s\)-line would
force \(\Delta u^2=0\), and hence \(u=r=s=0\), which is impossible in
projective space.

The common pole line \(u=0\) meets the conic transversely at
\(P_\infty^\pm=[0:1:\pm1]\), and every \(f_i\) has a simple pole at
both points.  Set
\[
  h=1+q+q^2+q^3=(q+1)(q^2+1).
\]
The resulting local exponents are
\[
\begin{array}{c|c|c}
 \text{intersection} & \text{number} & e_P\pmod d\\ \hline
 L_0\cap\mathscr C_\Delta &2&1\\
 L_1\cap\mathscr C_\Delta &2&q^3\\
 L_2\cap\mathscr C_\Delta &2&q^2\\
 L_3\cap\mathscr C_\Delta &2&q\\
 \{u=0\}\cap\mathscr C_\Delta&2&-h.
\end{array}
\]
As \(\gcd(d,q)=1\), the four zero lines contribute eight finite points
to \(\Sigma_f\); the two points at infinity contribute exactly when
\(d\nmid h\).  By
\Cref{prop:root-conic}, \(f_0\) is finite and nonzero on
\(\mathscr C_\Delta(\F_q)\), so the conic lemma applies.

\begin{theorem}[General character-sum bound]\label{thm:general-char}
If \(\chi\neq1\) and \(d=\ord\chi\), then
\[
  |S_\chi|\leq
  \begin{cases}
    6\sqrt q,&d\mid(q+1)(q^2+1),\\
    8\sqrt q,&d\nmid(q+1)(q^2+1).
  \end{cases}
\]
\end{theorem}

\begin{proof}
The preceding divisor calculation gives \(R_f=8\) when \(d\mid h\)
and \(R_f=10\) otherwise.  Apply \Cref{thm:conic-sum}; the factor
\(\chi(\mu)\) does not affect the absolute value.
\end{proof}

\subsection{Two lines on the quotient conic and the 2/4 estimate}

Suppose now that \(d=\ord\chi\) divides \(q^2-1\).  Then
\[
  \chi=\psi\circ\Nm_{\F_{q^4}/\F_{q^2}}
\]
for a unique character \(\psi\) of \(\F_{q^2}^\times\) of order \(d\).
Consider the quotient conic
\[
  \overline{\mathscr C}_\Delta:
  \quad X^2-\Delta^2Y^2=\Delta^2U^2
  \quad\subset\PP^2_{[U:X:Y]}.
\]
On the chart \(U\neq0\), set
\[
  \bar f_0=\frac{(\frac12-\Delta)U-X}{U},
  \qquad
  \bar f_1=\bar f_0^{[q]}
      =\frac{(\frac12+\Delta)U-X}{U}.
\]
These are the functions obtained by taking the relative norms of the
two \(\iota\)-pairs above.  More precisely,
\begin{equation}\label{eq:function-norm-descent}
  \bar f_0\circ\pi_\Delta=2f_0f_2,
  \qquad
  \bar f_1\circ\pi_\Delta=2f_1f_3.
\end{equation}
The first identity follows from \eqref{eq:norm-on-quotient} after
dividing by \(\mu^2\), and the second is its \(q\)-Frobenius conjugate.
Thus, on \(\F_q\)-points, the field norm of the actual root
\(\mu b(P)\) is the value of the descended function
\((\mu^2/2)\bar f_0\) on the quotient.
Let \(\eta\) and \(\varepsilon\) be the quadratic characters of
\(\F_{q^2}^\times\) and \(\F_q^\times\), respectively.

For \(P=[1:w]\in\mathscr C_\Delta(\F_q)\), the invariant
\(X\)-coordinate is
\[
  X=w^2+w^{2q}=2w^2-\Delta\in\F_q.
\]
For fixed \(X\in\F_q\), the number of such source points is
\[
  1-\eta\left(\frac{X+\Delta}{2}\right):
\]
indeed, \(z=(X+\Delta)/2\) satisfies \(z-z^q=\Delta\), and its two
square roots solve \eqref{eq:normalized-root-equation} and lie in
\(V_-\) exactly when \(z\) is a nonsquare in \(\F_{q^2}\).

On the target conic, the number of points with the same affine
\(X\)-coordinate is
\[
  1+\varepsilon\left(\frac{X^2-\Delta^2}{\Delta^2}\right).
\]
The quadratic characters satisfy
\(\eta(z)=\varepsilon(\Nm_{\F_{q^2}/\F_q}(z))\), and
\(\varepsilon(\Delta^2)=\Delta^{q-1}=-1\).  Therefore
\[
  1+\varepsilon\left(\frac{X^2-\Delta^2}{\Delta^2}\right)
  =1-\varepsilon(X^2-\Delta^2)
  =1-\eta\left(\frac{X+\Delta}{2}\right),
\]
so the two fiber sizes agree.  The nonsquare \(\Delta^2\) also prevents
the target conic from having an \(\F_q\)-point on \(U=0\).
Combining the fiber comparison with
\eqref{eq:norm-on-quotient} gives
\begin{equation}\label{eq:norm-sum-on-conic}
  S_\chi
  =\psi\left(\frac{\mu^2}{2}\right)
   \sum_{P\in\overline{\mathscr C}_\Delta(\F_q)}
      \psi\bigl(\bar f_0(P)\bigr).
\end{equation}

We apply \Cref{thm:conic-sum} to the sum on the right with
\[
  C=\overline{\mathscr C}_\Delta,\qquad m=2,
  \qquad f_i=\bar f_i\quad(0\leq i<2).
\]
The zero divisors of \(\bar f_0\) and \(\bar f_1\) are cut out by the
two lines
\[
\begin{aligned}
  \bar L_0&:\ X=\left(\frac12-\Delta\right)U,\\
  \bar L_1&:\ X=\left(\frac12+\Delta\right)U.
\end{aligned}
\]
Equation \eqref{eq:function-norm-descent} makes the pairing of the four
source lines explicit:
\[
  \operatorname{div}_0(\bar f_0\circ\pi_\Delta)
    =\operatorname{div}_0(f_0)+\operatorname{div}_0(f_2),
  \qquad
  \operatorname{div}_0(\bar f_1\circ\pi_\Delta)
    =\operatorname{div}_0(f_1)+\operatorname{div}_0(f_3).
\]
Thus \(\bar L_0\) descends the pair \(L_0,L_2\), while \(\bar L_1\)
descends the pair \(L_1,L_3\).
On these two lines, the conic equation becomes, respectively,
\[
  \Delta^2Y^2
  =\left(\frac14-\Delta\right)U^2,
  \qquad
  \Delta^2Y^2
  =\left(\frac14+\Delta\right)U^2.
\]
Since \(\Delta\notin\F_q\), both coefficients are nonzero.  Hence each
line meets the conic transversely in two affine geometric points, and
the two pairs are disjoint.  Their common pole divisor is cut out by
\(U=0\), which also meets the conic in two distinct points.  At the
latter points \(X\neq0\), so both \(\bar f_0\) and \(\bar f_1\) have simple
poles.  The local exponents are
\[
\begin{array}{c|c|ccc}
 \text{intersection} & \text{number}
 &\ord(\bar f_0)&\ord(\bar f_1)&e_P\pmod d\\ \hline
 \bar L_0\cap\overline{\mathscr C}_\Delta&2&1&0&1\\
 \bar L_1\cap\overline{\mathscr C}_\Delta&2&0&1&q\\
 \{U=0\}\cap\overline{\mathscr C}_\Delta&2&-1&-1&-(q+1).
\end{array}
\]
Thus norm descent replaces the four zero lines on
\(\mathscr C_\Delta\) by two zero lines on
\(\overline{\mathscr C}_\Delta\), with two intersections per line.  The
four finite points give the coefficient \(2\), while the two points at
infinity are exactly what raises it to \(4\).

\begin{theorem}[Characters factoring through the norm]
\label{thm:norm-char}
If \(\chi\neq1\) and \(d=\ord\chi\mid q^2-1\), then
\[
  |S_\chi|\leq
  \begin{cases}
    2\sqrt q,&d\mid q+1,\\
    4\sqrt q,&d\nmid q+1.
  \end{cases}
\]
\end{theorem}

\begin{proof}
The prefactor in \eqref{eq:norm-sum-on-conic} has absolute value \(1\).
The function \(\bar f_0\) is finite and nonzero on
\(\overline{\mathscr C}_\Delta(\F_q)\): there is no rational point at
infinity, while a rational point on \(\bar L_0\) would require
\(X/U=1/2-\Delta\notin\F_q\).  The table gives \(R_{\bar f}=4\) when
\(d\mid q+1\) and \(R_{\bar f}=6\) otherwise.  Apply
\Cref{thm:conic-sum} to \eqref{eq:norm-sum-on-conic}.
\end{proof}

\subsection{Bounds according to character order}

Define three disjoint sets of odd primes:
\begin{equation}\label{eq:prime-classes}
\begin{aligned}
  \Acal&=\{p\text{ an odd prime}:p\mid q-1\},\\
  \Bcal&=\{p\text{ an odd prime}:p\mid q+1\},\\
  \Ccal&=\{p\text{ an odd prime}:p\mid q^2+1\}.
\end{aligned}
\end{equation}
They are disjoint because the three displayed integers have pairwise
greatest common divisor \(2\).  From now on, \(e\) denotes an exact
character order, while \(d\) is reserved for a divisor indexing a
sieve coefficient.

Let \(\chi\) have exact squarefree order \(e>1\) dividing \(q^4-1\).
The condition \(\supp(e)\cap\Ccal=\varnothing\) is equivalent to
\(e\mid q^2-1\), since \(2\mid q^2-1\); it is therefore exactly the
condition under which the norm estimate applies.  Combining
\Cref{thm:general-char,thm:norm-char} gives
\begin{equation}\label{eq:u}
  |S_\chi|\leq u(e)\sqrt q,\qquad
  u(e)=2
  +2\one_{\{\exists\,p\in\Acal:\,p\mid e\}}
  +4\one_{\{\exists\,p\in\Ccal:\,p\mid e\}}.
\end{equation}
Equivalently,
\[
\begin{array}{c|cc}
u(e)
  & \supp(e)\cap\Ccal=\varnothing
  & \supp(e)\cap\Ccal\neq\varnothing\\
\hline
\supp(e)\cap\Acal=\varnothing    & 2 & 6\\
\supp(e)\cap\Acal\neq\varnothing & 4 & 8.
\end{array}
\]
Primes in \(\Bcal\), as well as the prime \(2\), do not change
\(u(e)\).  In particular, \(u(e)\leq8\); the refined sieve retains all
four values in \eqref{eq:u}.  These bounds are uniform in
\(\mu\in\F_q^\times\) and \(\Delta\in L^\times\).

\section{The double-core prime sieve}\label{sec:sieve}

Fix \(\mu\in\F_q^\times\) and \(\Delta\in L^\times\).  We now combine
the character-sum bounds with a lower-bound sieve.  We
first give a general weighted formulation, then introduce an exact
outer core and specialize the remaining weight to the Bagger--Punch
weight.  From this point on, all geometric information is compressed
into the four-valued function \(u(e)\) in \eqref{eq:u}.  Since those
bounds are uniform in \(\mu\) and \(\Delta\), so are all numerical
criteria derived below.

Let
\[
  n=q^4-1,\qquad R_n=\rad(n).
\]
For a prime \(p\mid n\) and \(P\in\mathscr C_\Delta(\F_q)\), define
\[
  \rho_p(P)=
  \one_{\{\mu b(P)\in(\F_{q^4}^{\times})^p\}}.
\]
For \(d\mid R_n\), let
\[
  \rho_d(P)=\prod_{p\mid d}\rho_p(P),\qquad \rho_1(P)=1.
\]
Thus \(\rho_d(P)\) is the indicator of
\(\mu b(P)\in(\F_{q^4}^{\times})^d\).  Character orthogonality gives
\begin{equation}\label{eq:rhod}
  \rho_d(P)=\frac1d\sum_{\chi^d=1}\chi(\mu b(P)).
\end{equation}
Since \(\F_{q^4}^{\times}\) is cyclic of order \(n\), an element is
primitive
if and only if it is not a \(p\)-th power for every prime \(p\mid n\).
Accordingly,
\begin{equation}\label{eq:primitive-indicator}
  \one_{\{\mu b(P)\text{ is primitive}\}}
  =\prod_{p\mid R_n}(1-\rho_p(P))
  =\sum_{d\mid R_n}\mu(d)\rho_d(P).
\end{equation}
Expanding this exact indicator introduces character sums of every
exact order \(e\mid R_n\).  A lower sieve uses a sparser coefficient
family, reducing the total error at the cost of a smaller main term.

\subsection{General lower-bound sieves}

Let \(\Pcal\) be a set of prime divisors of \(R_n\), and write
\[
  \Pcal^\#:=\prod_{p\in\Pcal}p.
\]
For real coefficients \(\lambda_d\), define
\begin{equation}\label{eq:general-w}
  w(P)=\sum_{d\mid\Pcal^\#}\lambda_d\rho_d(P).
\end{equation}
\begin{lemma}[Admissible lower weights]\label{lem:admissible}
If
\begin{equation}\label{eq:admissible}
  \lambda_1=1,\qquad
  \sum_{d\mid m}\lambda_d\leq0
  \quad(1<m\mid\Pcal^\#),
\end{equation}
then, for every \(P\in\mathscr C_\Delta(\F_q)\),
\[
  w(P)\leq\prod_{p\in\Pcal}(1-\rho_p(P))
 \].
\end{lemma}

\begin{proof}
For \(P\in\mathscr C_\Delta(\F_q)\), let
\[
  m(P)=\prod_{\substack{p\in\Pcal\\\rho_p(P)=1}}p.
\]
Then \(w(P)=\sum_{d\mid m(P)}\lambda_d\).  This is \(1\) when
\(m(P)=1\), and it is nonpositive by \eqref{eq:admissible} otherwise,
which proves the comparison.
\end{proof}

For a general lower weight, define
\begin{equation}\label{eq:B}
  B_e=\sum_{\substack{d\mid\Pcal^\#\\e\mid d}}\frac{\lambda_d}{d},
  \qquad e\mid\Pcal^\#.
\end{equation}

\begin{theorem}[General weighted lower-bound sieve]
\label{thm:linear-sieve}
Under \eqref{eq:admissible},
\begin{equation}\label{eq:general-sieve-inequality}
  \sum_{P\in\mathscr C_\Delta(\F_q)}w(P)
  \geq(q+1)B_1-\sqrt q
  \sum_{\substack{e\mid\Pcal^\#\\e>1}}\varphi(e)u(e)|B_e|.
\end{equation}
If \(B_1>0\), define
\[
  Z=\frac1{B_1}
  \sum_{\substack{e\mid\Pcal^\#\\e>1}}\varphi(e)u(e)|B_e|.
\]
Then \(q+1>Z\sqrt q\) is sufficient for
\(\sum_{P\in\mathscr C_\Delta(\F_q)}w(P)>0\).
\end{theorem}

\begin{proof}
Summing \eqref{eq:rhod} over \(P\in\mathscr C_\Delta(\F_q)\), inserting
\eqref{eq:general-w}, and interchanging the finite sums gives
\[
\begin{aligned}
  \sum_{P\in\mathscr C_\Delta(\F_q)}w(P)
 &=\sum_{d\mid\Pcal^\#}\frac{\lambda_d}{d}
   \sum_{\chi^d=1}S_\chi\\
 &=\sum_{e\mid\Pcal^\#}B_e
   \sum_{\ord\chi=e}S_\chi.
\end{aligned}
\]
If \(\ord\chi=e\), the character occurs in the \(d\)-sum exactly when
\(e\mid d\), so its coefficient is \(B_e\).  For \(e=1\),
the inner sum is \(S_1=q+1\).  For \(e>1\), there are
\(\varphi(e)\) characters of exact order \(e\), and
\eqref{eq:u} gives
\[
 \left|B_e\sum_{\ord\chi=e}S_\chi\right|
 \leq\varphi(e)u(e)|B_e|\sqrt q.
\]
Summing these estimates proves \eqref{eq:general-sieve-inequality}.
By \Cref{lem:admissible}, positivity of the left-hand side implies
that there is \(P\in\mathscr C_\Delta(\F_q)\) such that
\(\rho_p(P)=0\) for every
\(p\in\Pcal\).
Dividing by \(B_1>0\) gives the stated sufficient condition.
\end{proof}

\subsection{Exact outer cores}

For squarefree \(m\), the M\"obius coefficients have total mass
\[
  \sum_{d\mid m}\frac{\mu(d)}d=\frac{\varphi(m)}m.
\]
Using these coefficients for all primes dividing \(R_n\) introduces an
error term for every nontrivial divisor of \(R_n\).  A sparse lower
weight has fewer error terms, but applying it to the smallest primes
also reduces the main coefficient substantially.  We therefore choose
a divisor \(K\mid R_n\), let \(\Pcal\) be the complementary set of
prime divisors, and write \(R_n=K\Pcal^\#\).  The primes in \(K\),
typically including \(2\), are treated exactly, while the lower weight
is applied only to the primes in \(\Pcal\).  A related core decomposition appears in
\cite[Proposition~4.3]{CohenLines}.  The primitive indicator factors as
\[
 \one_{\{\mu b(P)\text{ is primitive}\}}
 =
 \left(\sum_{a\mid K}\mu(a)\rho_a(P)\right)
 \left(\sum_{d\mid\Pcal^\#}\mu(d)\rho_d(P)\right).
\]
We retain the \(K\)-factor exactly and replace the complementary factor
by an admissible lower weight \(w\), obtaining
\begin{equation}\label{eq:core-weight}
  w_K(P):=
  \prod_{p\mid K}(1-\rho_p(P))\,w(P).
\end{equation}
We call \(K\) the \emph{exact outer core}; here ``exact'' refers to the
complete M\"obius expansion over its prime divisors.  In the
Bagger--Punch weight below, the distinguished prime \(h\) is the inner
core of the double-core prime sieve.

Define
\begin{equation}\label{eq:EF}
  E(K)=\sum_{\substack{a\mid K\\a>1}}u(a),
  \qquad
  F_f(K)=\sum_{a\mid K}u(af)\quad(f\mid\Pcal^\#,\ f>1),
\end{equation}
and
\[
  \phicore(K)=\prod_{p\mid K}\left(1-\frac1p\right).
\]

\begin{proposition}[Exact-core reduction]\label{prop:core}
If \(B_1>0\), the weight \eqref{eq:core-weight} gives
\begin{equation}\label{eq:core-lower}
  N(n)\geq
  \phicore(K)B_1\bigl(q+1-Z_K\sqrt q\bigr),
\end{equation}
where
\begin{equation}\label{eq:ZK}
  Z_K=E(K)+\frac1{B_1}
  \sum_{\substack{f\mid\Pcal^\#\\f>1}}
  \varphi(f)F_f(K)|B_f|.
\end{equation}
In particular,
\begin{equation}\label{eq:core-sufficient}
  q+1>Z_K\sqrt q
\end{equation}
is sufficient for \(N(n)>0\).
\end{proposition}

\begin{proof}
Expanding \eqref{eq:core-weight} gives
\[
 w_K(P)=
 \left(\sum_{s\mid K}\mu(s)\rho_s(P)\right)
 \left(\sum_{d\mid\Pcal^\#}\lambda_d\rho_d(P)\right).
\]
By \Cref{lem:admissible},
\[
  w_K(P)\leq\one_{\{\mu b(P)\text{ is primitive}\}},
\]
and hence
\[
  N(n)\geq\sum_{P\in\mathscr C_\Delta(\F_q)}w_K(P).
\]
Since \((K,\Pcal^\#)=1\), we have
\(\rho_s(P)\rho_d(P)=\rho_{sd}(P)\).  Using \eqref{eq:rhod} and
grouping characters by their exact orders gives
\[
\begin{aligned}
  \sum_{P\in\mathscr C_\Delta(\F_q)}w_K(P)
 &=\sum_{s\mid K}\sum_{d\mid\Pcal^\#}
   \frac{\mu(s)\lambda_d}{sd}
   \sum_{\chi^{sd}=1}S_\chi\\
 &=\sum_{a\mid K}\sum_{f\mid\Pcal^\#}
   C_aB_f\sum_{\ord\chi=af}S_\chi,
\end{aligned}
\]
where
\[
 C_a
 =\sum_{\substack{s\mid K\\a\mid s}}\frac{\mu(s)}s
 =\frac{\mu(a)}a
  \prod_{p\mid K/a}\left(1-\frac1p\right),
 \qquad
 \varphi(a)|C_a|=\phicore(K).
\]

The term \(a=f=1\) contributes
\(\phicore(K)B_1(q+1)\).  Since
\(\varphi(af)=\varphi(a)\varphi(f)\), the remaining terms and
\eqref{eq:u} give
\[
\begin{aligned}
  \sum_{P\in\mathscr C_\Delta(\F_q)}w_K(P)
 &\geq\phicore(K)B_1(q+1)
  -\phicore(K)B_1E(K)\sqrt q\\
 &\quad
  -\phicore(K)\sqrt q
   \sum_{\substack{f\mid\Pcal^\#\\f>1}}
    \varphi(f)F_f(K)|B_f|\\
 &=\phicore(K)B_1\bigl(q+1-Z_K\sqrt q\bigr).
\end{aligned}
\]
Together with
\(N(n)\geq\sum_{P\in\mathscr C_\Delta(\F_q)}w_K(P)\), this proves
\eqref{eq:core-lower}.  The sufficient condition follows.
\end{proof}

For the present problem, we take \(K=2\).  This treats the condition
for \(p=2\) exactly, while adjoining the factor \(2\) to a character
order does not change the value of \(u\), since \(u(2f)=u(f)\).

\begin{corollary}[The core \(K=2\)]\label{cor:core2}
Let \(K=2\) and \(\Pcal^\#=R_n/2\).  For any admissible lower weight
\(w(P)=\sum_{d\mid\Pcal^\#}\lambda_d\rho_d(P)\) with \(B_1>0\),
\[
  N(n)\geq\frac{B_1}{2}\bigl(q+1-Z_2\sqrt q\bigr),
\]
where
\[
  Z_2=2+\frac{2}{B_1}
  \sum_{\substack{f\mid\Pcal^\#\\f>1}}\varphi(f)u(f)|B_f|.
\]
In particular, \(q+1>Z_2\sqrt q\) is sufficient for \(N(n)>0\).
\end{corollary}

\begin{proof}
Since \(2\) belongs to neither \(\Acal\) nor \(\Ccal\),
\(u(2f)=u(f)\).  Consequently,
\[
  E(2)=2,\qquad F_f(2)=u(f)+u(2f)=2u(f),
  \qquad \phicore(2)=\frac12.
\]
Substitution in \Cref{prop:core} gives the result.
\end{proof}

\subsection{The Bagger--Punch weight}

For \(w\) in \eqref{eq:core-weight}, we use the Bagger--Punch weight
with a single-prime core
\cite[Theorem~1 and Lemmas~6--7]{BaggerPunch}.  Fix
\(h\in\Pcal\), choose \(\Rcal\subseteq\Pcal\setminus\{h\}\), and let
\[
  \mathcal L=\Pcal\setminus(\{h\}\cup\Rcal).
\]
In the present notation, the weight is
\begin{equation}\label{eq:w2}
\begin{split}
  w_{h,\Rcal}(P)
  &=(1-\rho_h(P))
      \left(1-\sum_{r\in\Rcal}\rho_r(P)\right)
      -\sum_{\ell\in\mathcal L}\rho_\ell(P)\\
  &=1-\sum_{p\in\Pcal}\rho_p(P)
    +\rho_h(P)\sum_{r\in\Rcal}\rho_r(P).
\end{split}
\end{equation}

We verify admissibility directly.  If
\(\rho_p(P)=0\) for every \(p\in\Pcal\), both
\(w_{h,\Rcal}(P)\) and
\(\prod_{p\in\Pcal}(1-\rho_p(P))\) equal \(1\).  Otherwise, if
\(\rho_h(P)=1\), the first line of \eqref{eq:w2} gives
\[
 w_{h,\Rcal}(P)=-\sum_{\ell\in\mathcal L}\rho_\ell(P)\leq0;
\]
if \(\rho_h(P)=0\), the second line gives
\[
 w_{h,\Rcal}(P)=1-\sum_{p\in\Pcal}\rho_p(P)\leq0.
\]
Thus \(w_{h,\Rcal}\) is an admissible lower weight.

To apply \Cref{prop:core}, write \(w_{h,\Rcal}\) in the form
\eqref{eq:general-w}.  Its coefficients are
\[
 \lambda_d=
 \begin{cases}
  \mu(d),
   &d\in\{1\}\cup\Pcal\cup\{hr:r\in\Rcal\},\\
  0,&\text{otherwise}.
 \end{cases}
\]
Substitution in \eqref{eq:B} gives
\[
 B_1(h,\Rcal)
  =1-\sum_{p\in\Pcal}\frac1p
   +\sum_{r\in\Rcal}\frac1{hr}.
\]

Assume
\(\sum_{p\in\Pcal\setminus\{h\}}1/p<1\).  Then
\[
  B_h=-\frac1h\left(1-\sum_{r\in\Rcal}\frac1r\right)<0.
\]
If \(B_1(h,\Rcal)>0\), substituting the nonzero coefficients \(B_f\)
into \Cref{prop:core} gives the following coefficient of \(\sqrt q\):
\begin{equation}\label{eq:modified-error}
\begin{aligned}
 T_K(h,\Rcal)
 ={}&E(K)B_1(h,\Rcal)\\
 &+\frac{h-1}{h}
   \left(1-\sum_{r\in\Rcal}\frac1r\right)F_h(K)\\
 &+\sum_{r\in\Rcal}\frac{(h-1)(r-1)}{hr}
   \bigl(F_r(K)+F_{hr}(K)\bigr)\\
 &+\sum_{\ell\in\mathcal L}
   \left(1-\frac1\ell\right)F_\ell(K).
\end{aligned}
\end{equation}

This calculation and \Cref{prop:core} give the following criterion.

\begin{proposition}[Double-core prime sieve inequality]
\label{prop:modified}
Assume
\(\sum_{p\in\Pcal\setminus\{h\}}1/p<1\)
and \(B_1(h,\Rcal)>0\).  Then
\begin{equation}\label{eq:modified-sieve}
 N(n)\geq\phicore(K)
 \left\{B_1(h,\Rcal)(q+1)
       -T_K(h,\Rcal)\sqrt q\right\}.
\end{equation}
The corresponding normalized ratio is
\begin{equation}\label{eq:modified-ratio}
  Z_K(h,\Rcal)=
  \frac{T_K(h,\Rcal)}
       {B_1(h,\Rcal)},
\end{equation}
and
\[
  \frac{q+1}{\sqrt q}>Z_K(h,\Rcal)
\]
is sufficient for \(N(n)>0\).
\end{proposition}

\begin{remark*}
For \(K=2\), the reciprocal-sum hypothesis holds automatically when
\(q\leq2\times10^{11}\).  If \(\lvert\Pcal\rvert\leq27\), take
\(h=\min\Pcal\).  Then
\[
 \sum_{p\in\Pcal\setminus\{h\}}\frac1p
 \leq
 \sum_{\substack{5\leq p\leq107\\p\ \mathrm{prime}}}\frac1p
 <0.99844<1.
\]
Thus failure requires \(\lvert\Pcal\rvert\geq28\).  Since \(q\) is
odd, \(16\mid q^4-1\), and hence
\[
\begin{aligned}
 q^4-1
  &\geq
    16\prod_{\substack{3\leq p\leq109\\p\ \mathrm{prime}}}p
    >2\times10^{45},\\
 q&>2\times10^{11}.
\end{aligned}
\]
The hypothesis may also hold beyond this range.
\end{remark*}

\subsection{Optimization of the sieve weight}

For each \(h\in\Pcal\) satisfying the reciprocal-sum hypothesis in
\Cref{prop:modified}, we seek to minimize \(Z_K(h,\Rcal)\) over
\(\Rcal\).  The resulting minima can then be compared over \(h\).
Fix such an \(h\), and define
\[
  \delta_0=1-\sum_{p\in\Pcal}\frac1p.
\]
For the empty choice \(\Rcal=\varnothing\), we have
\(B_1=\delta_0\) and \(T_K=T_0(K)\), where
\[
 T_0(K)=E(K)\delta_0+
   \sum_{p\in\Pcal}\left(1-\frac1p\right)F_p(K).
\]
For \(r\in\Pcal\setminus\{h\}\), set
\[
\begin{aligned}
 \kappa_K(h,r)
 ={}&E(K)-(h-1)F_h(K)-(r-1)F_r(K)\\
    &+(h-1)(r-1)F_{hr}(K).
\end{aligned}
\]
Then
\begin{equation}\label{eq:modified-data}
\begin{split}
 B_1(h,\Rcal)
  &=\delta_0+\sum_{r\in\Rcal}\frac1{hr},\\
 T_K(h,\Rcal)
  &=T_0(K)+\sum_{r\in\Rcal}
        \frac{\kappa_K(h,r)}{hr}.
\end{split}
\end{equation}
Adding \(r\) to \(\Rcal\) changes \(B_1\) and \(T_K\) by
\[
  \Delta B_1=\frac1{hr},\qquad
  \Delta T_K=\frac{\kappa_K(h,r)}{hr},\qquad
  \frac{\Delta T_K}{\Delta B_1}=\kappa_K(h,r).
\]
Consequently, adding \(r\) decreases the current ratio
\(Z=T_K/B_1\) if and only if
\[
  \kappa_K(h,r)<Z.
\]

\begin{lemma}[Prefix selection]\label{lem:prefix}
Assume \(\delta_0>0\), and write
\(\Pcal\setminus\{h\}=\{r_1,\ldots,r_s\}\), where
\[
 \kappa_K(h,r_1)\leq\cdots\leq\kappa_K(h,r_s).
\]
Put
\[
  \kappa_i=\kappa_K(h,r_i),\qquad a_i=\frac1{hr_i},
\]
and, for \(0\leq j\leq s\), let
\[
 Z_j=
 \frac{T_0(K)+\displaystyle\sum_{i=1}^{j}\kappa_i a_i}
      {\delta_0+\displaystyle\sum_{i=1}^{j}a_i},
\]
where the sums are empty when \(j=0\).  Let \(j_\ast\) be the smallest
\(j\in\{0,\ldots,s\}\) such that either \(j=s\) or
\[
  \kappa_{j+1}\geq Z_j.
\]
Then
\[
  Z_{j_\ast}
  =\min_{\Rcal\subseteq\Pcal\setminus\{h\}}Z_K(h,\Rcal),
\]
and the minimum is attained by
\(\Rcal_\ast=\{r_i:1\leq i\leq j_\ast\}\).  Indices satisfying
\(\kappa_i=Z_{j_\ast}\) may be included or omitted arbitrarily without
altering the minimum.
\end{lemma}

\begin{proof}
Since \(\delta_0>0\), \(B_1(h,\Rcal)>0\) for every \(\Rcal\).
Choose a minimizing set \(\Rcal\), and put
\(Z_\ast=Z_K(h,\Rcal)\).  For \(r_i\notin\Rcal\),
\[
 Z_K(h,\Rcal\cup\{r_i\})-Z_\ast
 =\frac{a_i(\kappa_i-Z_\ast)}
        {B_1(h,\Rcal)+a_i},
\]
whereas for \(r_i\in\Rcal\),
\[
 Z_K(h,\Rcal\setminus\{r_i\})-Z_\ast
 =\frac{a_i(Z_\ast-\kappa_i)}
        {B_1(h,\Rcal)-a_i}.
\]
Both denominators are positive: the first is immediate, while for
\(r_i\in\Rcal\),
\[
 B_1(h,\Rcal)-a_i
 =\delta_0+\sum_{\substack{r_k\in\Rcal\\k\neq i}}a_k>0.
\]
The minimality of \(\Rcal\) therefore implies
\[
  r_i\in\Rcal\Longrightarrow\kappa_i\leq Z_\ast,
  \qquad
  r_i\notin\Rcal\Longrightarrow\kappa_i\geq Z_\ast.
\]
Thus every \(r_i\) with \(\kappa_i<Z_\ast\) belongs to \(\Rcal\), and
every \(r_i\) with \(\kappa_i>Z_\ast\) is omitted; indices with
\(\kappa_i=Z_\ast\) may be added or removed without changing the
ratio.  In particular, if \(r_j\) is selected and \(i<j\), then
\(\kappa_i\leq\kappa_j\leq Z_\ast\).  If
\(\kappa_i<Z_\ast\), the preceding implication gives
\(r_i\in\Rcal\); if \(\kappa_i=Z_\ast\), we may add \(r_i\) without
changing the minimum.  Thus a minimizing set may be enlarged,
without changing its value, to an initial segment.

For consecutive initial segments,
\[
 Z_{j+1}-Z_j
 =\frac{a_{j+1}(\kappa_{j+1}-Z_j)}
         {\delta_0+\displaystyle\sum_{i=1}^{j+1}a_i}.
\]
Hence \(Z_{j+1}<Z_j\) exactly when \(\kappa_{j+1}<Z_j\).
If \(\kappa_{j+1}\geq Z_j\), then \(Z_{j+1}\) is a weighted average of
\(Z_j\) and \(\kappa_{j+1}\), so
\[
  Z_j\leq Z_{j+1}\leq\kappa_{j+1}.
\]
If \(j+1<s\), then
\(\kappa_{j+2}\geq\kappa_{j+1}\geq Z_{j+1}\), so the next step cannot
decrease the ratio either.  Induction shows that no later step can
decrease it.  By the definition of \(j_\ast\), \(Z_j\) decreases
strictly for \(j<j_\ast\) and is nondecreasing for \(j\geq j_\ast\).
Thus \(Z_{j_\ast}\) is the minimum among the initial segments and,
by the first part of the proof, among all subsets.  Finally, adding
or removing any index with \(\kappa_i=Z_{j_\ast}\) leaves the ratio
equal to \(Z_{j_\ast}\).
\end{proof}

For \(K=2\), \Cref{cor:core2} gives
\begin{equation}\label{eq:kappa-core2}
\begin{split}
 \kappa_2(h,r)
  ={}&2-2(h-1)u(h)-2(r-1)u(r)\\
    &+2(h-1)(r-1)u(hr).
\end{split}
\end{equation}
For fixed \(h\) with \(\delta_0>0\), \Cref{lem:prefix} therefore
shows that the minimum over \(\Rcal\) is attained by an initial segment
after the elements of \(\Pcal\setminus\{h\}\) are ordered by
\(\kappa_2(h,r)\).

\subsection{An example}

We illustrate the selection rule with \(q=7109\).  The relevant
factorizations are
\[
\begin{aligned}
 q-1&=2^2\cdot1777,\\
 q+1&=2\cdot3^2\cdot5\cdot79,\\
 q^2+1&=2\cdot2273\cdot11117.
\end{aligned}
\]
Here \(1777\), \(2273\), and \(11117\) are prime, and hence
\[
 \begin{aligned}
  \Pcal&=\{3,5,79,1777,2273,11117\},\\
  \Acal&=\{1777\},\qquad
  \Bcal=\{3,5,79\},\qquad
  \Ccal=\{2273,11117\}.
 \end{aligned}
\]
Take \(K=2\).  For the empty prefix,
\[
 \delta_0\approx0.45,\qquad
 Z_0\approx94.30>84.33\approx\frac{q+1}{\sqrt q},
\]
so the empty choice does not prove positivity.

Now fix \(h=3\).  Since
\(u(3)=u(5)=u(15)=2\), \eqref{eq:kappa-core2} gives
\[
 \kappa_2(3,5)=10,\qquad
 \kappa_2(3,79)=306.
\]
These are the two smallest values of \(\kappa_2(3,r)\).  Since
\(10<Z_0\), the first step
selects \(5\).  For \(\Rcal=\{5\}\),
\[
 B_1(3,\{5\})\approx0.52,\qquad
 Z_2(3,\{5\})\approx83.49
 <84.33\approx\frac{q+1}{\sqrt q}.
\]
Since \(83.49<306\), the stopping rule in \Cref{lem:prefix} gives
\(\Rcal=\{5\}\) for \(h=3\).  The corresponding weight is
\[
 w_2(P)=(1-\rho_2(P))
 \left(1-\sum_{p\in\Pcal}\rho_p(P)+\rho_{15}(P)\right).
\]
Thus \Cref{prop:modified} gives \(N(n)>0\).

\begin{remark*}
For comparison, minimizing the ratio in the Bagger--Punch sieve
\cite[Theorem~1]{BaggerPunch} with the same exact-order bounds gives
\[
 Z_{\mathrm{BP}}\approx85.46
 >84.33\approx\frac{q+1}{\sqrt q}.
\]
Thus the original Bagger--Punch sieve does not prove positivity for
this example.  The exact outer core lowers the ratio from \(85.46\) to
\(83.49\), which is below the threshold \(84.33\).
\end{remark*}

\section{Finite verification}\label{sec:algorithm}

The uniformity in \(\mu\) and \(\Delta\) reduces the remaining proof to
arithmetic in \(q\).  We first derive a factorization-free cutoff from
\(u(e)\leq8\), then apply exact rational sieves below that cutoff.  The
residual values, together with the odd prime powers \(q\leq43\), are
settled by the exact residue-cover criterion.

\subsection{A uniform cutoff}

\begin{proposition}[Uniform cutoff]\label{prop:cutoff}
Every odd prime power
\[
  q\geq Q_0:=13288681
\]
satisfies \(N(n)>0\) for every
\((\mu,\Delta)\in\F_q^\times\times L^\times\).
\end{proposition}

\begin{proof}
Let \(\nu=\omega(q^4-1)\).  Since \(v_2(q^4-1)\geq4\),
\[
  q^4-1\geq8\rad(q^4-1).
\]
Let \(p_j\) denote the \(j\)-th prime and write
\(p_\nu^\#=\prod_{j=1}^{\nu}p_j\).  Then
\[
  q^4-1\geq8p_\nu^\#.
\]

We first dispose of large \(\nu\) by taking the full core \(K=R_n\).
Then \(\Pcal=\varnothing\) and \(B_1=1\), so
\Cref{prop:core} and the uniform estimate give
\[
  Z_K=E(K)\leq8(2^\nu-1).
\]
It is therefore enough to require
\[
  \sqrt q>8\cdot2^\nu.
\]
Indeed, this gives
\(q>8\cdot2^\nu\sqrt q>Z_K\sqrt q\).
In view of \(q^4-1\geq8p_\nu^\#\), this condition is implied by
\[
  8p_\nu^\#>(8\cdot2^\nu)^8,
  \qquad\text{or equivalently}\qquad
  p_\nu^\#>2^{8\nu+21}.
\]
An exact integer check shows that this last inequality fails at
\(\nu=160\) but holds at \(\nu=161\), where it reads
\(p_{161}^\#>2^{1309}\).  Since \(p_{161}=947>256\), every subsequent
prime also exceeds \(256\).  Thus, if the inequality holds at \(\nu\),
then
\[
  p_{\nu+1}^\#
  =p_{\nu+1}p_\nu^\#
  >2^8\,2^{8\nu+21}
  =2^{8(\nu+1)+21}.
\]
Thus the primorial inequality holds for every \(\nu\geq161\), and
\[
  q^4>8p_\nu^\#>(8\cdot2^\nu)^8.
\]
Hence \(\sqrt q>8\cdot2^\nu\), so every \(\nu\geq161\) is covered.

For \(1\leq\nu\leq160\), let \(K\) be an exact core with \(m\) prime
divisors, and use only the uniform bound.  Let
\(\lvert\Pcal\rvert=s\geq1\), and define
\[
  \delta=1-\sum_{p\in\Pcal}\frac1p.
\]
If \(\delta>0\), \Cref{prop:modified} applies with any
\(h\in\Pcal\) and \(\Rcal=\varnothing\).  In this case,
\[
  B_1(h,\varnothing)=\delta.
\]
Since \(\mathcal L=\Pcal\setminus\{h\}\),
\eqref{eq:modified-error} simplifies to
\[
  T_K(h,\varnothing)
  =E(K)\delta+
    \sum_{p\in\Pcal}\left(1-\frac1p\right)F_p(K).
\]
The uniform estimate gives
\[
\begin{aligned}
  E(K)&\leq8(2^m-1),&
  F_p(K)&\leq8\cdot2^m,\\
  \sum_{p\in\Pcal}\left(1-\frac1p\right)&=s-1+\delta.
\end{aligned}
\]
Consequently,
\[
\begin{aligned}
  Z_K(h,\varnothing)
  &\leq8(2^m-1)
    +\frac{8\cdot2^m}{\delta}(s-1+\delta)\\
  &=8\left\{2^m\left(2+\frac{s-1}{\delta}\right)-1\right\}.
\end{aligned}
\]
Thus \Cref{prop:modified} gives the sufficient condition
\begin{equation}\label{eq:uniform-cutoff}
  q+1>
  8\sqrt q\left\{2^m\left(2+\frac{s-1}{\delta}\right)-1\right\}.
\end{equation}

We now apply \eqref{eq:uniform-cutoff} to \(q^4-1\).  Let
\(r_1<\cdots<r_\nu\) be its distinct prime divisors.  The primorial
bound gives
\[
  q\geq
  L_\nu:=\left\lceil(8p_\nu^\#+1)^{1/4}\right\rceil.
\]
For \(0\leq m<\nu\), take \(K=r_1\cdots r_m\), with \(K=1\) when
\(m=0\).  The remaining set has \(s=\nu-m\) primes.  Since
\(r_j\geq p_j\),
\[
  \delta
  =1-\sum_{j=m+1}^{\nu}\frac1{r_j}
  \geq
  1-\sum_{j=m+1}^{\nu}\frac1{p_j}
  =:\delta_{\nu,m}.
\]
The bound in \eqref{eq:uniform-cutoff} decreases as \(\delta\)
increases.  Hence, whenever \(\delta_{\nu,m}>0\), it is enough to
check
\[
  \frac{L_\nu+1}{\sqrt{L_\nu}}>
  8\left\{
    2^m\left(2+\frac{\nu-m-1}{\delta_{\nu,m}}\right)-1
  \right\}.
\]
Since \((x+1)/\sqrt x\) is increasing for \(x>1\), this check covers
every \(q\geq L_\nu\).  Clearing denominators and squaring reduces it
to an integer comparison.  A direct computation with exact rational
arithmetic in SageMath~10.9 shows that the displayed inequality holds
with \(m=3\) for \(21\leq\nu\leq68\) and with \(m=5\) for
\(69\leq\nu\leq160\).

For \(1\leq\nu\leq20\), we minimize the integer threshold from
\eqref{eq:uniform-cutoff} over all \(m\) with
\(\delta_{\nu,m}>0\).  \Cref{tab:uniform-thresholds} records the exact
minima.  Their maximum occurs at \(\nu=20\), with \(m=3\), and equals
\[
  Q_0=13288681.
\]
This covers the remaining values of \(\nu\) and proves the proposition.
\end{proof}

\subsection{The finite sieve computation}

It remains to consider the odd prime powers \(43<q<Q_0\).
For each such \(q\), SageMath~10.9~\cite{SageMath} factors
\(q-1\), \(q+1\), and \(q^2+1\).  Since
\(n=(q-1)(q+1)(q^2+1)\), these factorizations determine the prime
divisors of \(n\) and their classes \(\Acal,\Bcal,\Ccal\).  The
quantities \(B_1\) and \(Z\) are then computed in \(\mathbb Q\).  A
weight with \(B_1>0\) certifies \(q\) if
\[
  \frac{q+1}{\sqrt q}>Z.
\]

For comparison, we optimize the Cohen and Bagger--Punch sieves in the
forms of \cite[Proposition~4.3]{CohenLines} and
\cite[Theorem~1 and Lemmas~6--7]{BaggerPunch}, respectively.

The double-core criterion is evaluated independently for every \(q\)
in the interval.  We take \(K=2\).  For each \(h\) with
\(\delta_0>0\), \Cref{lem:prefix} reduces the choice of \(\Rcal\) to
the initial segments determined by \(\kappa_2(h,r)\).  We evaluate
\(Z_2(h,\Rcal)\) for every such \(h\) and initial segment.  The exact
SageMath computation is summarized in \Cref{tab:finite-computation}.

\begin{table}[htbp]
\centering
\caption{Exact finite verification for \(43<q<Q_0\).}
\label{tab:finite-computation}
\vspace{0.6em}
\small
\renewcommand{\arraystretch}{1.05}
\setlength{\tabcolsep}{1.2em}
\begin{tabular}{@{\hspace{1em}}lrrr@{\hspace{1em}}}
\toprule
Sieve & Certified & Unresolved & Largest unresolved\\
\midrule
Cohen & \(863{,}748\) & \(3{,}631\) & \(278{,}881\)\\
Bagger--Punch & \(864{,}434\) & \(2{,}945\) & \(278{,}881\)\\
Double-core & \(864{,}675\) & \(2{,}704\) & \(204{,}931\)\\
\bottomrule
\end{tabular}
\end{table}

\subsection{The projective-conic verification}

The remaining finite cases are the \(2{,}704\) residual values in
\Cref{tab:finite-computation}, together with the odd prime powers at
most \(43\).  Fix \(\Delta\in L^\times\) and take the
normalized root function
\(b([1:w])=w-1/2\).  By \eqref{eq:root-set-scaling}, this conic
certifies every multiplier precisely when
\begin{equation}\label{eq:geometric-multiplier-cover}
  \forall\,\mu\in\F_q^\times\quad
  \exists\,P\in\mathscr C_\Delta(\F_q)
  \quad\text{such that \(\mu b(P)\) is primitive.}
\end{equation}

Let \(g\) generate \(\F_{q^4}^\times\), and write
\[
  n=q^4-1,\qquad
  M=\frac{n}{q-1}=(q+1)(q^2+1),\qquad
  \gamma=g^M.
\]
Then \(\gamma\) generates \(\F_q^\times\), so every multiplier is
\(\gamma^j\).  For \(P\in\mathscr C_\Delta(\F_q)\), write
\(b(P)=g^{\kappa(P)}\), with \(0\leq\kappa(P)<n\).

\begin{lemma}[Exact residue-cover criterion]\label{lem:residue-cover}
For \(j\in\mathbb Z\), the element \(\gamma^j b(P)\) is primitive if
and only if
\[
\begin{cases}
  2\nmid\kappa(P),\quad
  p\nmid\kappa(P)&(p\in\Bcal\cup\Ccal),\\
  j\not\equiv-4^{-1}\kappa(P)\pmod p&(p\in\Acal).
\end{cases}
\]
Consequently, \eqref{eq:geometric-multiplier-cover} is equivalent to
\begin{equation}\label{eq:exact-cover}
  \bigcup_{\substack{P\in\mathscr C_\Delta(\F_q)\\
                     2\nmid\kappa(P)\\
                     p\nmid\kappa(P)\ (p\in\Bcal\cup\Ccal)}}
       \prod_{p\in\Acal}
       \left(\F_p\setminus\{-4^{-1}\kappa(P)\}\right)
  =\prod_{p\in\Acal}\F_p.
\end{equation}
Here the excluded value in the \(p\)-th factor is taken modulo \(p\).
\end{lemma}

\begin{proof}
\[
  \gamma^j b(P)=g^{\kappa(P)+Mj}.
\]
This element is primitive precisely when no prime divisor of \(n\)
divides \(\kappa(P)+Mj\).  Since \(4\mid M\), the condition at \(2\)
is \(2\nmid\kappa(P)\).  If \(p\in\Bcal\cup\Ccal\), then \(p\mid M\),
whereas if \(p\in\Acal\), then
\[
  M=(q+1)(q^2+1)\equiv4\pmod p,
\]
which gives the displayed conditions.  As \(j\) varies modulo \(q-1\),
the Chinese remainder theorem shows that
\((j\bmod p)_{p\in\Acal}\) runs through
\(\prod_{p\in\Acal}\F_p\), proving \eqref{eq:exact-cover}.
\end{proof}

The rational points of \(\mathscr C_\Delta\) admit an explicit
parametrization by the solutions \(t\in\F_{q^4}\) of
\(t^{q+1}=\Delta\):
\begin{equation}\label{eq:computational-conic-bijection}
  \bigl\{t\in\F_{q^4}:t^{q+1}=\Delta\bigr\}
  \xrightarrow{\ \sim\ }\mathscr C_\Delta(\F_q),
  \qquad
  t\longmapsto\left[1:\frac{t+t^q}{2}\right].
\end{equation}
The inverse sends \([1:w]\) to \(w-w^q\).  Indeed, if
\(P=[1:w]\in\mathscr C_\Delta(\F_q)\), put \(t=w-w^q\).  Then
\[
  t^q=w^q+w,
  \qquad
  t^{q+1}=(w-w^q)(w+w^q)=\Phi(w)=\Delta,
  \qquad
  w=\frac{t+t^q}{2}.
\]
Conversely, suppose that \(t^{q+1}=\Delta\).  Since
\(\Delta^q=-\Delta\) and \(t\neq0\),
\[
  t^{q^2-1}
  =\frac{(t^{q+1})^q}{t^{q+1}}
  =\frac{\Delta^q}{\Delta}
  =-1,
\]
so \(t^{q^2}=-t\).  Put \(w=(t+t^q)/2\).  Then \(w\in V_-\) and
\[
  w-w^q=t,
  \qquad
  w+w^q=t^q,
  \qquad
  \Phi(w)=t^{q+1}=\Delta.
\]
Thus \([1:w]\in\mathscr C_\Delta(\F_q)\), and the two displayed
constructions are inverse.

We verified \eqref{eq:exact-cover} exhaustively in SageMath~10.9
\cite{SageMath}, using \eqref{eq:computational-conic-bijection} to
generate the conic points.  The test succeeds for every
\(\Delta\in L^\times\) for all \(2{,}704\) odd prime powers left
unresolved by the double-core computation.  It also succeeds for every
\(\Delta\in L^\times\) for each odd prime power \(q\leq43\) outside
\(\mathcal E\).  With the multiplier fixed at \(\mu=1\), it succeeds
for every \(\Delta\in L^\times\) for all odd prime powers \(q\leq43\),
except \(q=13\).

For completeness, we record one failed cover for each exceptional
value of \(q\).  Use the generators
\(g\in\F_{q^4}^\times\) and \(\zeta\in\F_{q^2}^\times\) selected by
the computation, and retain \(\gamma=g^M\).  Put
\[
  t_0=g^{(q^2+1)/2},
  \qquad
  \Delta_i=(\zeta^i t_0)^{q+1}.
\]
The following table gives a value of \(\Delta_i\) and an uncovered
multiplier \(\mu\) for each exceptional \(q\).

\begin{table}[H]
\centering
\caption{Counterexamples from the exact finite verification.}
\label{tab:counterexamples}
\vspace{0.6em}
\begin{tabular}{c@{\qquad}c@{\qquad}c}
\toprule
\(q\) & \(\Delta\) & \(\mu\)\\
\midrule
 7  & \(\Delta_0\)  & \(\gamma\)\\
11  & \(\Delta_0\)  & \(\gamma^4\)\\
13  & \(\Delta_0\)  & \(1\)\\
19  & \(\Delta_3\)  & \(\gamma\)\\
29  & \(\Delta_3\)  & \(\gamma^5\)\\
31  & \(\Delta_{10}\) & \(\gamma^2\)\\
41  & \(\Delta_{10}\) & \(\gamma^3\)\\
43  & \(\Delta_2\)  & \(\gamma^7\)\\
\bottomrule
\end{tabular}
\end{table}

In each row, \eqref{eq:exact-cover} fails at the displayed \(\Delta\),
and the displayed \(\mu\) is an uncovered multiplier.  Let
\(\bar\alpha\) be the unique coset such that, for \(\alpha\in\bar\alpha\),
\[
  \alpha-\alpha^q=\mu^2\Delta.
\]
Then \Cref{lem:residue-cover} and \eqref{eq:root-set-scaling} show that
\(\mathcal F_{\mu,\bar\alpha}\) has no primitive member.

\begin{proposition}[Exact finite verification]
\label{prop:finite-verification}
For every one of the \(2{,}704\) residual odd prime powers, and for
every odd prime power \(q\leq43\) with \(q\notin\mathcal E\), the exact
cover \eqref{eq:exact-cover} holds for every \(\Delta\in L^\times\).
Moreover, if \(q\leq43\) and \(q\neq13\), then for every
\(\Delta\in L^\times\) there is a point
\(P\in\mathscr C_\Delta(\F_q)\) for which \(b(P)\) is primitive.
\end{proposition}

\subsection{Completion of the proof}

\begin{proof}[Proof of \Cref{thm:main}]
Let \(q\notin\mathcal E\), and fix
\(\mu\in\F_q^\times\) and a coset
\(\bar\alpha\in\F_{q^2}/\F_q\).  For \(\bar\alpha\neq\F_q\), combine
\Cref{prop:cutoff}, the finite sieve computation summarized in
\Cref{tab:finite-computation}, and \Cref{prop:finite-verification}.
Together they give \(N(n)>0\), so
\Cref{cor:reduction} supplies the required primitive member.

For \(\bar\alpha=\F_q\), the assertion follows directly from Cohen's
prescribed-trace theorem \cite[Theorem~1]{CohenTrace}.
\end{proof}

\begin{proof}[Proof of \Cref{cor:GM}]
Gow and McGuire proved
\[
  \text{Conjecture~2}\quad\Longleftrightarrow\quad
  \text{Conjecture~3}.
\]
Conjectures~2 and~3 therefore follow immediately from \Cref{thm:main},
and the \(q=43\) row of \Cref{tab:counterexamples} proves sharpness.
Conjecture~1 is the specialization \(\mu=1\): \Cref{thm:main} covers
\(q>43\), and the fixed-multiplier computation in
\Cref{prop:finite-verification} covers \(q\leq43\) except \(q=13\).
\end{proof}

\section{Conclusion}\label{sec:conclusion}

For every odd prime power \(q\notin\mathcal E\), every family
\(\mathcal F_{\mu,\bar\alpha}\) contains a primitive member.
In particular, Gow and McGuire's Conjecture~1 holds for \(q\neq13\),
while Conjectures~2 and~3 hold for \(q>43\), with a sharp threshold.

The proof is governed by a genus-zero geometry.  The split and
anisotropic restrictions of \(\Phi\) give two projective conics whose
affine points are the reducible and irreducible roots.  On the latter,
\(q^2\)-Frobenius acts through the root-pair involution.  The relative
norm of the root function descends to the quotient conic, pairing four
source lines into two quotient lines.  This explains the bounds
\(6\sqrt q,8\sqrt q\) before descent and
\(2\sqrt q,4\sqrt q\) after descent.

In higher degrees, it is natural to seek a variety whose rational points
parametrize the normalized roots and whose Frobenius strata record
factorization.  The cubic case is the first test of whether a Galois
quotient of such a variety can also carry norm descent.

\appendix
\renewcommand{\thesection}{Appendix~\Alph{section}}

\section{Verification of the uniform cutoff}
\label{app:uniform-cutoff}

We record the exact finite comparisons used to optimize the cutoff for
\(1\leq\nu\leq20\) in the proof of \Cref{prop:cutoff}.  Retain the
notation \(L_\nu\) and \(\delta_{\nu,m}\) from that proof, and write
\[
  Z_{\nu,m}
  =8\left\{
    2^m\left(2+\frac{\nu-m-1}{\delta_{\nu,m}}\right)-1
   \right\}.
\]

Let
\[
  \mathcal M_\nu
  =\left\{m\in\mathbb Z:
    0\leq m<\nu,\quad
    \delta_{\nu,m}
    =1-\sum_{j=m+1}^{\nu}\frac1{p_j}>0
   \right\}.
\]
The value \(m=\nu\) is excluded because it leaves no primes outside the
core and is covered by the full-core argument in the proof.
For \(m\in\mathcal M_\nu\), let \(Q_{\nu,m}\) be the least integer
\(q>1\) for which
\[
  \frac{q+1}{\sqrt q}>Z_{\nu,m},
\]
and define
\[
  Q_\nu=\min_{m\in\mathcal M_\nu}Q_{\nu,m}.
\]

The minimizing value of \(m\) is determined without numerical
approximation.  Since
\[
  \delta_{\nu,m+1}
  =\delta_{\nu,m}+\frac1{p_{m+1}},
\]
direct subtraction gives
\[
\begin{aligned}
 Z_{\nu,m+1}-Z_{\nu,m}
 ={}&
 \frac{8\cdot2^m}
      {\delta_{\nu,m}\delta_{\nu,m+1}}\\
 &{}\times
 \left[
  2\delta_{\nu,m}^2
  +\left(\nu-m-3+\frac2{p_{m+1}}\right)\delta_{\nu,m}
  -\frac{\nu-m-1}{p_{m+1}}
 \right].
\end{aligned}
\]
For every adjacent pair in each admissible range, clearing the positive
denominators in the bracket gives an integer sign comparison.  These
comparisons show that \(Z_{\nu,m}\) decreases up to the value \(m_\nu\)
listed below and increases thereafter.  Thus \(m_\nu\) is the unique
minimizer.  The value \(Q_\nu\) is then determined from its defining
inequality.

\begin{table}[H]
\centering
\caption{Exact optimized thresholds for \(1\leq\nu\leq20\).}
\label{tab:uniform-thresholds}
\vspace{0.6em}
\small
\renewcommand{\arraystretch}{1.05}
\setlength{\tabcolsep}{1.6em}
\begin{tabular}{@{\hspace{1em}}cccc@{\hspace{1em}}}
\toprule
\(\nu\) & admissible \(m\) & \(m_\nu\) & \(Q_\nu\)\\
\midrule
 1 & \(m=0\)          & 0 & \(62\)\\
 2 & \(0\leq m\leq1\)  & 1 & \(574\)\\
 3 & \(1\leq m\leq2\)  & 2 & \(3{,}134\)\\
 4 & \(1\leq m\leq3\)  & 2 & \(10{,}960\)\\
 5 & \(1\leq m\leq4\)  & 2 & \(28{,}569\)\\
 6 & \(1\leq m\leq5\)  & 2 & \(63{,}600\)\\
 7 & \(1\leq m\leq6\)  & 2 & \(124{,}846\)\\
 8 & \(1\leq m\leq7\)  & 2 & \(229{,}863\)\\
 9 & \(1\leq m\leq8\)  & 2 & \(397{,}153\)\\
10 & \(2\leq m\leq9\)  & 2 & \(644{,}694\)\\
11 & \(2\leq m\leq10\) & 2 & \(1{,}025{,}198\)\\
12 & \(2\leq m\leq11\) & 2 & \(1{,}569{,}908\)\\
13 & \(2\leq m\leq12\) & 3 & \(2{,}261{,}672\)\\
14 & \(2\leq m\leq13\) & 3 & \(3{,}057{,}777\)\\
15 & \(2\leq m\leq14\) & 3 & \(4{,}056{,}467\)\\
16 & \(2\leq m\leq15\) & 3 & \(5{,}275{,}078\)\\
17 & \(2\leq m\leq16\) & 3 & \(6{,}745{,}191\)\\
18 & \(2\leq m\leq17\) & 3 & \(8{,}553{,}225\)\\
19 & \(2\leq m\leq18\) & 3 & \(10{,}707{,}175\)\\
20 & \(2\leq m\leq19\) & 3 & \(13{,}288{,}681\)\\
\bottomrule
\end{tabular}
\end{table}

In particular,
\[
  \max_{1\leq\nu\leq20}Q_\nu
  =Q_{20}=13288681,
\]
which is the value \(Q_0\) used in \Cref{prop:cutoff}.

\clearpage
%\section{SageMath code}\label{app:source-code}

%The complete SageMath implementation used in the finite verification is
%given below.  It consists of the shared exact-arithmetic routines, the three
%sieve computations, the combined finite verification, and the verification
%of the uniform cutoff, and the normalized projective-conic verification.

%\subsection{Shared exact-arithmetic routines}

%\lstinputlisting[style=sagesource]{sage_verification/sieve_common.py}

%\subsection{Cohen sieve}

%\lstinputlisting[style=sagesource]{sage_verification/verify_cohen.py}

%\subsection{Bagger--Punch sieve}

%\lstinputlisting[style=sagesource]{sage_verification/verify_bagger_punch.py}

%\subsection{Double-core sieve}

%\lstinputlisting[style=sagesource]{sage_verification/verify_double_core.py}

%\subsection{Combined finite verification}

%\lstinputlisting[style=sagesource]{sage_verification/verify_staged.py}

%\subsection{Verification of the uniform cutoff}

%\lstinputlisting[style=sagesource]{sage_verification/verify_uniform_cutoff.py}

%\subsection{Normalized projective-conic verification}

%\lstinputlisting[style=sagesource]{sage_verification/direct_verify_normalized.py}

\end{document}